\documentclass{amsart}
\usepackage{amssymb,latexsym,enumitem,graphics,tikz}

\theoremstyle{plain}
\newtheorem*{thma}{Theorem A}
\newtheorem{thm}{Theorem}

\newtheorem{lemma}[thm]{Lemma}
\newtheorem{prop}[thm]{Proposition}

\theoremstyle{definition}

\newtheorem{ex}[thm]{Example}

\theoremstyle{remark}

\numberwithin{equation}{section}

\begin{document}
\title[Multi-Opponent James Functions]
{Multi-Opponent James Functions}
\author[C. N. B. Hammond and W. P. Johnson]{Christopher N. B. Hammond and Warren P. Johnson}

\date{June 10, 2015}
\address{Department of Mathematics\\
Connecticut College\\
New London, CT 06320}
\email{cnham@conncoll.edu}
\email{wpjoh@conncoll.edu}

\maketitle
\thispagestyle{empty}

\begin{abstract}
The James function, also known as the ``log5 method," assigns a probability to the result of a competition between two teams based on their respective winning percentages.
This paper, which builds on earlier work of the authors and Steven J.\ Miller, explores the analogous situation where a single team or player competes simultaneously against multiple opponents.
\end{abstract}

\section{Introduction}\label{S:intro}

In his 1981 \textit{Baseball Abstract} \cite{james}, Bill James posed the following problem:
suppose two teams $A$ and $B$ have winning percentages $a$ and $b$ respectively, having played
equally strong schedules in a game such as baseball where there are
no ties.  If $A$ and $B$ play each other, what is the probability $P(a,b)$ that $A$ wins?
James proposed a method of answering  this question (the so-called ``log5 method"), with an equivalent formulation given by Dallas 
Adams.

\begin{thma}
The probability that a team with winning percentage $a$ defeats a team with winning
percentage $b$ is given by the function
\begin{equation}\label{jamesfunction}
P(a,b)=\frac{a(1-b)}{a(1-b)+b(1-a)}\text{,}
\end{equation}
except when $a=b=0$ or $a=b=1$, in which case the probability is undefined.
\end{thma}

\noindent Adopting Adams's terminology, we refer to (\ref{jamesfunction}) as the \textit{James function}.  A proof of Theorem A was given much later by Steven J.\ Miller; see \cite[Theorem 1]{hjm}.  The purpose of the current paper is to extend this result to a multi-opponent game involving a protagonist $A$ and $n$ opponents
$B_{1},B_{2},\ldots,B_{n}$. In other words, we will find a formula for what we will call a \textit{multi-opponent James function}, or more specifically an \textit{$n$-opponent James function}.  In addition, we will study many interesting structural properties that hold for such functions for $n\geq 2$.

One of the distinctive aspects of the treatment in \cite{hjm} was an emphasis on the following properties, known as the \textit{James conditions}:
\begin{enumerate}[label=(\alph*)]
\item $P(a,\frac12)=a$.
\item $P(a,0)=1$ for $0<a\leq 1$.
\item $P(b,a)=1-P(a,b)$.
\item $P(1-b,1-a)=P(a,b)$.
\item $P(a,b)$ is a non-decreasing function of $a$ for $0\leq b\leq 1$ and a strictly increasing function of $a$ for $0<b<1$.
\end{enumerate}
These properties apply to all points $(a,b)$ with $0\leq a\leq 1$ and $0\leq b\leq 1$, except for $(0,0)$ and $(1,1)$.
The James function satisfies all five of these conditions, as well as another that is worthy of attention:
 \begin{equation}\label{involutive}
 P(a,b)=c\hspace{.1in}\text{ if and only if }\hspace{.1in}P(a,c)=b\text{,}
\end{equation}
as long as $0<a<1$. This condition can be stated in several equivalent forms, including
 \begin{equation}\label{involutive2}
 P(a,b)=1-c\hspace{.1in}\text{ if and only if }\hspace{.1in}P(c,b)=1-a
\end{equation}
for $0<b<1$.

Any function that satisfies conditions (a) to (e) is known as a \textit{Jamesian function}.  Any Jamesian function
that satisfies (\ref{involutive}) is called an \textit{involutive Jamesian function}.  As we shall see, the multi-opponent James functions satisfy a similar set of conditions.

\section{The functions $P_{n}$}\label{S:formula}

Before we attempt to find a formula for the multi-opponent James functions, we will need to define more precisely what we mean by ``winning percentage"
in this context.  Even though our goal is to describe the outcome of a multi-opponent game, we will think of the winning percentage as relating to a competition between
only two players.  From a practical point of view, this quantity can be obtained by treating any game with $n+1$ competitors (that is, one protagonist and $n$ opponents) as $\binom{n+1}{2}$ separate single-opponent competitions.  In other words, a competitor that finishes third out of ten would be credited with two losses and seven wins.
Based on this interpretation, the functions we are about to describe are most directly applicable to competitions where there is minimal interaction among the
competitors.  For example, bowling, sprinting, swimming, and crew would all be reasonable candidates for this model; horse racing and short track skating would not.

We are now in a position to establish a formula for $P_{n}$, the $n$-opponent James function.  The question makes no sense if all the winning percentages are $0$
or at least two of them are $1$.  If exactly one of them is $1$, then that competitor always wins, so we can exclude
that case also.  Otherwise, we can model this situation in the same manner as in the proof of Theorem A.  In other words, independently
assign either a $0$ or a $1$ to each competitor, where $A$ draws $1$ with probability $a$ and each $B_{i}$
draws $1$ with probability $b_{i}$.  If exactly one competitor draws a $1$, then it is declared the winner,
and otherwise we repeat the procedure.  The protagonist $A$ wins on the first draw with probability 
$a(1-b_{1})(1-b_{2})\cdots(1-b_{n})$, and the probability that \emph{some} competitor wins on the first draw is
\begin{multline*}
a(1-b_{1})(1-b_{2})\cdots(1-b_{n})+(1-a)(1-b_{1})(1-b_{2})\cdots(1-b_n)\sum_{i=1}^{n}\frac{b_i}{1-b_i}\\
=\left(\prod_{i=1}^{n}(1-b_i)\right)\left(a+(1-a)\sum_{i=1}^{n}\frac{b_i}{1-b_i}\right)\text{.}
\end{multline*}
If we call this quantity $M$ for a moment, the probability $P_{n}(a;b_{1},b_{2}\ldots,b_{n})$ that $A$
wins must satisfy the functional equation
\[
P_{n}(a;b_{1},b_{2},\ldots,b_{n})=a(1-b_{1})(1-b_{2})\cdots(1-b_{n})
+\left(1-M)P_{n}(a;b_{1},b_{2},\ldots,b_{n}\right)\text{,}
\]
or
\[
M\,P_{n}(a;b_{1},b_{2},\dots,b_{n})=a(1-b_{1})(1-b_{2})\cdots(1-b_{n})\text{.}
\]
Since by assumption none of the $b_{i}$ is $1$, we may solve this equation to obtain the following result.

\begin{thm}
The probability that a protagonist with winning percentage $a$ defeats $n$ opponents with winning
percentages $b_{1}, b_{2}, \ldots, b_{n}$ is given by the function
\begin{equation}\label{multi1}
P_{n}(a;b_{1},b_{2},\ldots,b_{n})=\frac{a}{a+(1-a)\displaystyle\sum\limits_{i=1}^n\frac{b_i}{1-b_{i}}}\text{.}
\end{equation}
\end{thm}

There are several alternate representations for $P_{n}$, such as
\begin{equation}\label{multi2}
P_{n}(a;b_{1},b_{2},\ldots,b_{n})=\frac{a\displaystyle\prod_{i=1}^{n}(1-b_{i})}
{a\displaystyle\prod_{i=1}^{n}(1-b_{i})+\sum_{j=1}^{n}b_{j}(1-a)\prod_{\substack{i=1,\\ i\neq j}}^{n}(1-b_{i})}\text{,}
\end{equation}
which is an analog of (\ref{jamesfunction}).  Perhaps the most useful representation is
\begin{equation}\label{trig}
P_{n}(a;b_{1},b_{2},\ldots,b_{n})=\frac{q(a)}{q(a)+\displaystyle\sum_{i=1}^{n}q(b_{i})}\text{,}
\end{equation}
where
\begin{equation}\label{qform}
q(s)=\frac{s}{1-s}\text{.}
\end{equation}
It is worth noting that $q$ is a constant multiple of the ``log5 function" originally introduced by James.

As we mentioned earlier, we will call $P_{n}$ the \textit{$n$-opponent James function}, referring to $P=P_{1}$ as the \textit{single-opponent James function}
when appropriate.  Even though it was derived under the assumption that
none of the winning percentages is $1$, the function $P_{n}$ as stated in (\ref{multi2}) has the correct behavior when exactly one winning percentage
(either $a$ or one of the $b_{i}$) is $1$.  For the duration
of this paper, we will extend the definition of $P_{n}$ to include these cases.

Several important properties are apparent from the various formulae for $P_{n}$:
\begin{enumerate}[label=(\Alph*)]
\item $P_{n}(a;\frac{1}{n+1},\frac{1}{n+1},\ldots,\frac{1}{n+1})=a$.
\item $P_{n}(a;b_{1},b_{2},\ldots,b_{n-1},0)=P_{n-1}(a;b_{1},b_{2},\ldots,b_{n-1})$.
\item $\displaystyle\sum_{i=1}^{n}P_{n}(b_{i};a,b_{1},b_{2},\ldots,b_{i-1},b_{i+1},\ldots,b_{n})=1-P_{n}(a;b_{1},b_{2},\ldots,b_{n})$.
\item $P_{n}(\underbrace{1-b;1-b,\ldots,1-b}_{k\text{ terms}},\underbrace{1-a,1-a,\ldots,1-a}_{n+1-k\text{ terms}})=P_{n}(\underbrace{a;a,\ldots,a}_{k\text{ terms}},\underbrace{b,b,\ldots,b}_{n+1-k\text{ terms}})$ when $0<a<1$ and $0<b<1$, for any $1\leq k\leq n$.
\item $P_{n}(a;b_{1},b_{2},\ldots,b_{n})$ is a non-increasing function of $b_{1}$, and a
strictly decreasing function of $b_{1}$ when $0<a<1$ and $0\leq b_{i}<1$ for $2\leq i\leq n$.
\item Permuting the values $b_{1},b_{2},\ldots,b_{n}$ does not affect the value of $P_{n}(a;b_{1},b_{2},\ldots,b_{n})$.
\end{enumerate}
Condition (D) may be most easily obtained from (\ref{trig}), with the additional observation that $q(1-s)=1/q(s)$.
Note that condition (E) can be translated to any $b_{j}$ for $2\leq j\leq n$; thus it implies that
$P_{n}(a;b_{1},b_{2},\ldots,b_{n})$ is a non-decreasing function of $a$, and a
strictly increasing function of $a$ when $0\leq b_{i}<1$ for $1\leq i\leq n$ and at least
one $b_{i}$ is nonzero.

We will refer to this list as the \textit{multi-James conditions}.
Conditions (A) to (E) may be viewed as generalizations of James conditions (a) to (e) respectively.  (The relationship between (B) and
(b) can best be understood by thinking of the function $P_{0}$ as being identically $1$.)  Condition (F), which is new to the multi-opponent situation,
dictates that the order in which one considers the opponents is unimportant.

Note that (A) does not provide a necessary condition for $P_{n}(a;b_{1},b_{2},\ldots,b_{n})=~a$ when $n\geq 2$.  In particular,
the following proposition may be obtained from (\ref{trig}).

\begin{prop}
 If $\sum_{i=1}^{n}q(b_{i})=1$, then $P_{n}(a;b_{1},b_{2},\ldots,b_{n})=a$.  The converse holds when $0<a<1$.
 \end{prop}

The function $P_{n}$ also satisfies a generalization of (\ref{involutive2}), the alternate version of the involutive property.

\begin{prop}\label{pseudo}
As long as $0\leq b_{i}<1$ for all $1\leq i\leq n$ and at least one $b_{i}$ is nonzero, we have that
\[
P_{n}(a;b_{1},b_{2},\ldots,b_{n})=1-c\hspace{.1in}\text{ if and only if }\hspace{.1in}P_{n}(c;b_{1},b_{2},\ldots,b_{n})=1-a\text{.}
\]  
\end{prop}

\begin{proof}
Since the denominator in (\ref{multi1}) is guaranteed to be nonzero, the expression
\[
P_{n}(a;b_{1},b_{2},\ldots,b_{n})=1-c
\]
is equivalent to
\[
a=a(1-c)+(1-a)(1-c)\sum_{i=1}^{n}\frac{b_i}{1-b_i}
\]
and also
\[
ac=(1-a)(1-c)\sum_{i=1}^n\frac{b_i}{1-b_i}\text{.}
\]
The last equation is symmetric in $a$ and $c$, which proves our assertion.
\end{proof}

The treatment in \cite{hjm} made considerable use of the level curves for the single-opponent James function, as well as for other Jamesian functions.  While we will
not provide a complete description of the level sets for $P_{n}$, we make the following observation.

\begin{prop}\label{level}
The expression $P_{n}(a;b_{1},b_{2},\ldots,b_{n})$ is equal to
\[
P_{n}\!\left(\frac{ta}{1+(t-1)a};\frac{tb_{1}}{1+(t-1)b_{1}},\frac{tb_{2}}{1+(t-1)b_{2}},\ldots,\frac{tb_{n}}{1+(t-1)b_{n}}\right)
\]
for any real number $t>0$.
\end{prop}

\begin{proof}
This identity follows from (\ref{trig}), along with the observation that
\[
q\!\left(\frac{ts}{1+(t-1)s}\right)=tq(s)
\]
for any positive value of $t$.
\end{proof}

\bigskip

We conclude this section by observing two additional properties of $P_{n}$.  First of all, note that
\begin{align}
\nonumber\frac{P_{n}(b;a,c_{2},c_{3},\ldots,c_{n})}{P_{n}(a;b,c_{2},c_{3},\ldots,c_{n})}
&=\frac{\displaystyle\frac{q(b)}{q(b)+q(a)+\displaystyle\sum_{i=2}^{n}q(c_{i})}}{\displaystyle\frac{q(a)}{q(a)+q(b)+\displaystyle\sum_{i=2}^{n}q(c_{i})}}\\
&=\frac{q(b)}{q(a)}
=\frac{\displaystyle\frac{q(b)}{q(b)+q(a)}}{\displaystyle\frac{q(a)}{q(a)+q(b)}}=\frac{P(b,a)}{P(a,b)}\label{irrelevant}
\end{align}
for $n\geq 2$.  This property, which can be described as \textit{independence from irrelevant alternatives}, may be extended to include all cases where $0<a\leq 1$ and $0\leq b<1$,
as long as $0\leq c_{i}<1$ for $2\leq i\leq n$.  Similarly, consider the \textit{odds ratio} defined
in the following manner:
\[
\frac{\displaystyle\frac{P_{m}(a;c_{1},c_{2},\ldots,c_{m})}{1-P_{m}(a;c_{1},c_{2},\ldots,c_{m})}}
{\displaystyle\frac{P_{n}(a;b_{1},b_{2},\ldots,b_{n})}{1-P_{n}(a;b_{1},b_{2},\ldots,b_{n})}}
=\frac{P_{m}(a;c_{1},c_{2},\ldots,c_{m})\bigl(1-P_{n}(a;b_{1},b_{2},\ldots,b_{n})\bigr)}{\bigl(1-P_{m}(a;c_{1},c_{2},\ldots,c_{m})\bigr)P_{n}(a;b_{1},b_{2},\ldots,b_{n})}\text{,}
\]
where all the winning percentages belong to the interval $(0,1)$.  Note that the preceding expression may be rewritten
\[
\frac{\displaystyle\left(\frac{q(a)}{q(a)+\displaystyle\sum_{i=1}^{m}q(c_{i})}\right)
\displaystyle\left(\frac{\displaystyle\sum_{i=1}^{n}q(b_{i})}{q(a)+\displaystyle\sum_{i=1}^{n}q(b_{i})}\right)}
{\displaystyle\left(\frac{\displaystyle\sum_{i=1}^{m}q(c_{i})}{q(a)+\displaystyle\sum_{i=1}^{m}q(c_{i})}\right)
\displaystyle\left(\frac{q(a)}{q(a)+\displaystyle\sum_{i=1}^{n}q(b_{i})}\right)}
=\frac{\displaystyle\sum_{i=1}^{n}q(b_{i})}{\displaystyle\sum_{i=1}^{m}q(c_{i})}\text{.}
\]
In other words, as long as the same value of $a$ is plugged into both $P_{n}$ and $P_{m}$, the odds ratio is independent of $a$.  This property
can be viewed as a complement to independence from irrelevant alternatives, relating instead to independence from the protagonist.


\section{Relationship to the Bradley--Terry--Luce model}

One of the main issues discussed in \cite{hjm} is the connection between the James function and the Bradley--Terry model, one of
the fundamental tools in the theory of paired comparisons.  It is not surprising that the multi-opponent James functions have a similar affinity to
a corresponding extension of this model (due primarily to Luce \cite{luce}).
Let $T$ denote a finite collection of possible outcomes in a situation requiring a single
selection or choice.  Under certain assumptions (see Theorems 3 and 4 in \cite{luce}), Luce posits that the probability of an outcome $x$ being preferred over
the other elements of $T$ is given by the formula
\begin{equation}\label{lucemodel}
P_{T}(x)=\frac{v(x)}{\displaystyle\sum_{y\in T}v(y)}\text{,}
\end{equation}
where $v$ is a positive, real-valued function (which he calls a \textit{ratio scale}).  This representation is sometimes referred to as the
\textit{Bradley--Terry--Luce model} or the \textit{strict utility model}.

While (\ref{lucemodel}) is highly reminiscent of (\ref{trig}),
Luce conceives of the ratio scale $v$ in a somewhat different manner from our function $q$.  In particular, $v$ is a constant multiple of a probability
that depends on the situation one is considering,
rather than a predetermined function.  This difference reflects the fact that our approach assumes an \textit{a priori} notion
of the worth of a competitor, namely its winning percentage along with the corresponding value of $q$, whereas in Luce's construction
any notion of worth must be inferred from the outcomes of particular competitions.  Nevertheless,
there are several formal similarities between our results and Luce's, which we shall note explicitly.  In the next section, we will follow Luce's example
by restating the probability associated with a multi-opponent competition in terms of a set of single-opponent competitions.

\section{Formulae for $P_{n}$}\label{S:formulae}

We will now introduce several additional formulae for a multi-opponent James function $P_{n}$, primarily in terms of the single-opponent
James function $P=P_{1}$.  Throughout this section, unless otherwise stated,
we will assume that all the winning percentages in our propositions and formulae belong to the interval $(0,1)$.
This is not a major constraint, as a winning percentage of $0$ or $1$ would either determine the value of $P_{n}$ automatically or allow
$P_{n}$ to be reduced to $P_{n-1}$ (see also Lemma \ref{zerolem} below).

We begin with an elementary observation.  Since $P$ satisfies James condition (c), it follows that
\begin{equation}\label{funid}
\frac{P(b,a)}{P(a,b)}=\frac{1}{P(a,b)}-1
\end{equation}
for $0<a\leq 1$ and $0\leq b<1$.  For the remainder of this paper, we will use the quantities in (\ref{funid}) interchangeably.

We will now establish two fundamental identities.  Consider the representation for $P_{n}$ given in (\ref{trig}), with $q$ defined
as in (\ref{qform}).  It follows from (\ref{irrelevant}) that
\begin{align*}
\frac{1}{P_{n}(a;b_{1},b_{2},\ldots,b_{n})}-1&=\frac{q(a)+\displaystyle\sum_{i=1}^{n}q(b_{i})}
{q(a)}-\frac{q(a)}{q(a)}\\
&=\sum_{i=1}^{n}\frac{q(b_{i})}{q(a)}\\&=\sum_{i=1}^{n}\frac{P(b_{i},a)}{P(a,b_{i})}
\text{.}
\end{align*}
Hence we have obtained our first major identity, which is analogous to Theorem 1 in \cite{luce}.
\begin{prop}[\textbf{Sum Formula}]\label{sumform}
\begin{align*}
\frac{1}{P_{n}(a;b_{1},b_{2},\ldots,b_{n})}-1=&\sum_{i=1}^{n}\frac{P(b_{i},a)}{P(a,b_{i})}\\
=&\sum_{i=1}^{n}\left(\frac{1}{P(a,b_{i})}-1\right)
\end{align*}
for $n\geq 1$.
\end{prop}
The intermediate identity
\[
\frac{1}{P_{n}(a;b_{1},b_{2},\ldots,b_{n})}-1=\sum_{i=1}^{n}\frac{q(b_{i})}{q(a)}
\]
that arose during the proof of this formula is also worth recording.\bigskip

To obtain our second major identity, consider $0<c<1$ and note that
\begin{align*}
\frac{1}{P_{n}(a;b_{1},b_{2},\ldots,b_{n})}-1&=\sum_{i=1}^{n}\frac{q(b_{i})}{q(a)}\\
&=\frac{q(c)}{q(a)}\left(\sum_{i=1}^{n}\frac{q(b_{i})}{q(c)}\right)\\
&=\frac{P(c,a)}{P(a,c)}\left(\frac{1}{P_{n}(c;b_{1},b_{2},\ldots,b_{n})}-1\right)\text{.}
\end{align*}
In other words, we have obtained the following result.
\begin{prop}[\textbf{Substitution Formula}]\label{subform}
\begin{align*}
\frac{1}{P_{n}(a;b_{1},b_{2},\ldots,b_{n})}-1&=\frac{P(c,a)}{P(a,c)}\left(\frac{1}{P_{n}(c;b_{1},b_{2},\ldots,b_{n})}-1\right)\\
&=\left(\frac{1}{P(a,c)}-1\right)\left(\frac{1}{P_{n}(c;b_{1},b_{2},\ldots,b_{n})}-1\right)
\end{align*}
for $n\geq 1$.  
\end{prop}
\noindent When $n=1$, this formula is analogous to Theorem 2 in \cite{luce}, which Luce later referred to as the \textit{product rule} (see \cite{luce2}).
\bigskip

Both the sum formula and the substitution formula can be applied in a variety of different ways.  For example, 
partition the set $\{b_{1},b_{2},\ldots,b_{n}\}$ into $k$ nonempty disjoint subsets $\mathcal{S}_{1},\mathcal{S}_{2},\ldots,\mathcal{S}_{k}$,
with $m_{j}=|\mathcal{S}_{j}|$ for $1\leq j\leq k$.  Let $P_{m_{j}}(a;\mathcal{S}_{j})$ denote the probability of a
protagonist with winning percentage $a$ defeating all
$m_{j}$ opponents with winning percentages in the set $\mathcal{S}_{j}$.
The sum formula shows that
\[
\frac{1}{P_{m_{j}}(a;\mathcal{S}_{j})}-1=\sum_{b_{i}\in\mathcal{S}_{j}}\frac{P(b_{i},a)}{P(a,b_{i})}
\]
for each $1\leq j\leq k$.  Consequently
\begin{align}
\frac{1}{P_{n}(a;b_{1},b_{2},\ldots,b_{n})}-1&=\sum_{i=1}^{n}\frac{P(b_{i},a)}{P(a,b_{i})}\nonumber\\
&=\sum_{j=1}^{k}\sum_{b_{i}\in\mathcal{S}_{j}}\frac{P(b_{i},a)}{P(a,b_{i})}\nonumber\\
&=\sum_{j=1}^{k}\left(\frac{1}{P_{m_{j}}(a;\mathcal{S}_{j})}-1\right)\text{,}\label{Sform1}
\end{align}
which may be viewed as a generalization of the sum formula.\bigskip

Many useful identities may be derived from applying both the sum and substitution formulae.  Equation (\ref{Sform1})
may be rewritten
\[
\frac{1}{P_{n}(a;b_{1},b_{2},\ldots,b_{n})}-\frac{1}{P_{m_{1}}(a;\mathcal{S}_{1})}=\sum_{j=2}^{k}\left(\frac{1}{P_{m_{j}}(a;\mathcal{S}_{j})}-1\right)\text{.}
\]
Taking $\mathcal{S}_{1}=\{b_{1}\}$ and $\mathcal{S}_{2}=\{b_{2},b_{3},\ldots,b_{n}\}$, we obtain
\begin{align*}
\frac{1}{P_{n}(a;b_{1},b_{2},\ldots,b_{n})}-\frac{1}{P(a,b_{1})}=\frac{1}{P_{n-1}(a;b_{2},b_{3},\ldots,b_{n})}-1\text{.}
\end{align*}
Applying the substitution formula, we see that
\begin{align*}
\frac{1}{P_{n}(a;b_{1},b_{2},\ldots,b_{n})}-\frac{1}{P(a,b_{1})}&=\frac{P(b_{1},a)}{P(a,b_{1})}\left(\frac{1}{P_{n-1}(b_{1};b_{2},b_{3},\ldots,b_{n})}-1\right)\\
&=\frac{P(b_{1},a)}{P(a,b_{1})}\left(\frac{1}{P_{n-1}(b_{1};b_{2},b_{3},\ldots,b_{n})}\right)-\frac{P(b_{1},a)}{P(a,b_{1})}\text{,}
\end{align*}
from which we obtain another formula that will prove essential to our later results.
\begin{prop}[\textbf{Reduction Formula}]\label{reductioprop}
\[
\frac{1}{P_{n}(a;b_{1},b_{2},\ldots,b_{n})}-1=\frac{P(b_{1},a)}{P(a,b_{1})}\left(\frac{1}{P_{n-1}(b_{1};b_{2},b_{3},\ldots,b_{n})}\right)
\]
for $n\geq 2$.
\end{prop}
\noindent For future reference, note that this identity is still valid when $b_{i}=0$ for any or all $2\leq i\leq n$.  For $n=1$, the formula
reduces to (\ref{funid}) if we take $P_{0}$ to be identically $1$.

The reduction formula may in turn be rewritten
\[
\frac{1}{P_{n}(a;b_{1},b_{2},\ldots,b_{n})}-1=\frac{P(b_{1},a)}{P(a,b_{1})}\left(1+\left(\frac{1}{P_{n-1}(b_{1};b_{2},b_{3},\ldots,b_{n})}-1\right)\right)
\]
which, combined with the sum formula for $P_{n-1}$, yields another important formula.

\begin{prop}[\textbf{Shifted Sum Formula}]\label{shift}
\[
\frac{1}{P_{n}(a;b_{1},b_{2},\ldots,b_{n})}-1=\frac{P(b_{1},a)}{P(a,b_{1})}\left(1+\sum_{i=2}^{n}\frac{P(b_{i},b_{1})}{P(b_{1},b_{i})}\right)
\]
for $n\geq 2$.
\end{prop}

\noindent This particular representation for $P_{n}$ provides the basis for one of the major results in Section \ref{S:odds}.\bigskip

We may obtain additional formulae for $P_{n}$ by iterating the reduction formula.  For example,
\begin{align*}
\frac{1}{P_{n}(a,b_{1},b_{2},\ldots,b_{n})}-1&=\frac{P(b_{1},a)}{P(a,b_{1})}\left(1+\left(\frac{1}{P_{n-1}(b_{1};b_{2},b_{3},\ldots,b_{n})}-1\right)\right)\\
&=\frac{P(b_{1},a)}{P(a,b_{1})}\left(1+\frac{P(b_{2},b_{1})}{P(b_{1},b_{2})}
\left(\frac{1}{P_{n-2}(b_{2};b_{3},b_{4},\ldots,b_{n})}\right)\right)\\
&=\frac{P(b_{1},a)}{P(a,b_{1})}\left(1+\frac{P(b_{2},b_{1})}{P(b_{1},b_{2})}
\left(1+\sum_{i=3}^{n}\frac{P(b_{i},b_{2})}{P(b_{2},b_{i})}\right)\right)\\
\end{align*}
for $n\geq 3$.  Applying the reduction formula a total of $n-1$ times, we obtain the following representation.

\begin{prop}[\textbf{Expanded Sum Formula}]\label{expand}
\begin{align*}
\frac{1}{P_{n}(a;b_{1},b_{2},\ldots,b_n)}-1=&\sum_{j=1}^{n}\prod_{i=1}^j\frac{P(b_{i},b_{i-1})}{P(b_{i-1},b_{i})}\\
=&\sum_{j=1}^{n}\prod_{i=1}^j\left(\frac{1}{P(b_{i-1},b_{i})}-1\right)
\end{align*}
for $n\geq 1$, with the understanding that $b_{0}=a$.
\end{prop}
\bigskip

It makes sense for these various sum formulae to exist, as they imply that the probability associated with
a multi-opponent competition can be determined from
a certain number of single-opponent competitions, even when the winning percentages of the competitors are unknown.
What is perhaps surprising is that one only needs to consider $n$ such competitions.

The sum formula (Proposition \ref{sumform}) shows that the probability of a protagonist $A$ defeating $B_{1},B_{2},\ldots,B_{n}$
in a multi-opponent competition can be determined solely from the probabilities associated with the $n$ single-opponent competitions involving $A$.
The shifted sum formula (Proposition \ref{shift}) shows that the probability of $A$ defeating $B_{1},B_{2},\ldots,B_{n}$ can be determined solely from the probabilities associated with the $n$ single-opponent competitions involving any particular opponent $B_{i}$.
The expanded sum formula (Proposition \ref{expand}) demonstrates that the probability
of $A$ defeating $B_{1},B_{2},\ldots,B_{n}$ can be determined from the $n$ single-opponent
competitions $A$ versus $B_{1}$ and $B_{i}$ versus $B_{i+1}$ for $1\leq i\leq n-1$.  It is natural to wonder precisely what
combinations of $n$ single-opponent competitions are sufficient to determine $P_{n}(a;b_{1},b_{2},\ldots,b_{n})$.  As it turns out,
the answer to this question is as nice as one could possibly hope.  The following fact is a direct consequence of the
path formula (Theorem \ref{pathform}), which we shall prove momentarily.

\begin{thm}\label{tworep}
Suppose one knows the probabilities associated with $n$ single-opponent competitions
involving the competitors $A,B_{1},B_{2},\ldots,B_{n}$.  If every opponent $B_{i}$ may
be connected to $A$ by a sequence of these single-opponent competitions, then
these $n$ probabilities are sufficient to determine
the probability of  $A$ defeating $B_{1},B_{2},\ldots,B_{n}$.
\end{thm}

Let $G$ be a graph with vertex set $\{A,B_{1},B_{2},\ldots,B_{n}\}$, where there is an edge between
two vertices if and only if we know the probability associated with the single-opponent competition
between them.  The hypotheses of Theorem \ref{tworep} imply that $G$ is connected graph with $n+1$ vertices and $n$ edges, which means that $G$ is
a tree (see \cite[Theorem 4.8]{cz}), for which we can think of $A$ as being the root.  
Suppose the unique path from $A$ to $B_{i}$ has length $k$, say
\begin{equation}\label{path}
A\rightarrow B_{\ell_{1}}\rightarrow B_{\ell_{2}}\rightarrow\cdots\rightarrow B_{\ell_{k}}=B_{i}\text{.}
\end{equation}
Applying the substitution formula $k-1$ times, we see that
\[
\frac{P(b_{i},a)}{P(a,b_{i})}
=\frac{P(b_{\ell_{1}},a)}{P(a,b_{\ell_{1}})}
\frac{P(b_{\ell_{2}},b_{\ell_{1}})}{P(b_{\ell_{1}},b_{\ell_{2}})}\cdots
\frac{P(b_{\ell_{k-1}},b_{\ell_{k-2}})}{P(b_{\ell_{k-2}},b_{\ell_{k-1}})}
\frac{P(b_{i},b_{\ell_{k-1}})}{P(b_{\ell_{k-1}},b_{i})}
\]
for each $1\leq i\leq n$.  Hence the sum formula yields the following representation for $P_{n}$.

\begin{thm}[\textbf{Path Formula}]\label{pathform}
\[
\frac1{P_n(a;b_{1},b_{2},\ldots,b_n)}-1=\sum_{i=1}^{n}\frac{P(b_{\ell_{1}},a)}{P(a,b_{\ell_{1}})}
\frac{P(b_{\ell_{2}},b_{\ell_{1}})}{P(b_{\ell_{1}},b_{\ell_{2}})}\cdots
\frac{P(b_{\ell_{k-1}},b_{\ell_{k-2}})}{P(b_{\ell_{k-2}},b_{\ell_{k-1}})}
\frac{P(b_{i},b_{\ell_{k-1}})}{P(b_{\ell_{k-1}},b_{i})}\text{,}
\]
where the indices $\ell_{1},\ell_{2},\ldots,\ell_{k}$, which depend on $i$, are defined as in (\ref{path}).
\end{thm}

\noindent With the appropriate interpretation, the sum formula, the shifted sum formula, and the expanded sum formula
may all be viewed as specific instances of the path formula.

\begin{ex}\label{ex1}
To illustrate the path formula further, consider the tree in Figure 1.  Since $B_{1}$, $B_{2}$, and $B_{7}$
connect directly to $A$, they contribute the terms
\[
\frac{P(b_{1},a)}{P(a,b_{1})}\text{,} \quad \frac{P(b_{2},a)}{P(a,b_{2})}\text{,}
\quad\text{and}\quad \frac{P(b_{7},a)}{P(a,b_{7})}
\]
respectively to the sum.  Since $B_{3}$ and $B_{4}$ both connect to $A$ via $B_{2}$, together they contribute
\[
\frac{P(b_{2},a)}{P(a,b_{2})}\frac{P(b_{3},b_{2})}{P(b_{2},b_{3})}+
\frac{P(b_{2},a)}{P(a,b_{2})}\frac{P(b_{4},b_{2})}{P(b_{2},b_{4})}\text{.}
\]
Since $B_{8}$ connects to $A$ via $B_{7}$ it contributes
\[
\frac{P(b_{7},a)}{P(a,b_{7})}\frac{P(b_{8},b_{7})}{P(b_{7},b_{8})}\text{.}
\]
Finally, since $B_{5}$ and $B_{6}$ connect to $A$ via $B_{4}$ and $B_{2}$, they contribute a total of
\[
\frac{P(b_{2},a)}{P(a,b_{2})}\frac{P(b_{4},b_{2})}{P(b_{2},b_{4})}\frac{P(b_{5},b_{4})}{P(b_{4},b_{5})}
+\frac{P(b_{2},a)}{P(a,b_{2})}\frac{P(b_{4},b_{2})}{P(b_{2},b_{4})}\frac{P(b_{6},b_{4})}{P(b_{4},b_{6})}\text{.}
\]
Therefore
\begin{align*}
\frac{1}{P_{8}(a;b_{1},b_{2},\ldots,b_{8})}&-1\\
=\frac{P(b_{1},a)}{P(a,b_{1})}&+\frac{P(b_{2},a)}{P(a,b_{2})}+
\frac{P(b_{2},a)}{P(a,b_{2})}\frac{P(b_{3},b_{2})}{P(b_{2},b_{3})}+
\frac{P(b_{2},a)}{P(a,b_{2})}\frac{P(b_{4},b_{2})}{P(b_{2},b_{4})}\\
&+\frac{P(b_{2},a)}{P(a,b_{2})}\frac{P(b_{4},b_{2})}{P(b_{2},b_{4})}\frac{P(b_{5},b_{4})}{P(b_{4},b_{5})}
+\frac{P(b_{2},a)}{P(a,b_{2})}\frac{P(b_{4},b_{2})}{P(b_{2},b_{4})}\frac{P(b_{6},b_{4})}{P(b_{4},b_{6})}\\
&+\frac{P(b_{7},a)}{P(a,b_{7})}+\frac{P(b_{7},a)}{P(a,b_{7})}\frac{P(b_{8},b_{7})}{P(b_{7},b_{8})}
\text{.}
\end{align*}
\end{ex}

\bigskip

\begin{figure}[ht]\label{fig1}
\begin{tikzpicture}
 [scale=.8,auto=left,every node/.style={draw,circle,minimum size=1cm}]
  \node (n0) at (10,10) {$A$};
  \node (n1) at (7,7)  {$B_{1}$};
  \node (n2) at (10,7)  {$B_{2}$};
  \node (n3) at (7,4)  {$B_{3}$};
  \node (n4) at (10,4)  {$B_{4}$};
  \node (n5) at (8.5,1)  {$B_{5}$};
  \node (n6) at (11.5,1)  {$B_{6}$};
  \node (n7) at (13,7) {$B_{7}$};
  \node (n8) at (13,4) {$B_{8}$};
  \foreach \from/\to in {n0/n1,n0/n2,n2/n3,n2/n4,n4/n5,n4/n6,n0/n7,n7/n8}
    \draw (\from) -- (\to);
\end{tikzpicture}
\caption{The graph representing Example \ref{ex1}.}
\end{figure}
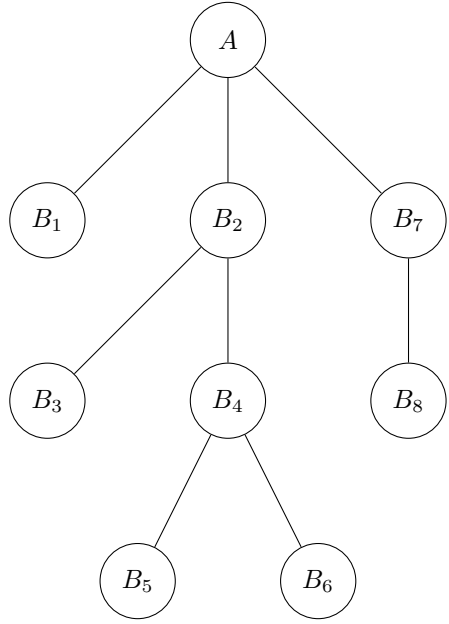

While the point of Theorem \ref{tworep} is that we do not need to know the winning
percentages of any of the competitors, the hypotheses of this theorem guarantee that one can
determine every competitor's winning percentage if one knows any single competitor's winning percentage.  Suppose, in the context
of Example \ref{ex1}, one knew the winning percentage $b_{4}$.  Since we know the value of $P(b_{4},b_{5})$,
the involutive property would allow us to determine $b_{5}=P(b_{4},P(b_{4},b_{5}))$.  Likewise, we
could determine both $b_{2}$ and $b_{6}$.  In this manner, one could use the winning percentage at
each vertex to find the winning percentages at all adjacent vertices, eventually determining every competitor's winning percentage.
\bigskip

We conclude this section with one last major identity, along with a few of its consequences.
Consider $P_{n}(a;b,c_{2},c_{3},\ldots,c_{n})$ and $P_{m}(b;a,d_{2},d_{3},\ldots,d_{m})$, where $a$ and $b$
both belong to the interval $(0,1)$. The reduction formula (Proposition \ref{reductioprop}) dictates that
\[
\frac{1}{P_{n}(a;b,c_{2},c_{3},\ldots,c_{n})}-1=\frac{1}{\delta_{1}}
\left(\frac{P(b,a)}{P(a,b)}\right)\text{,}
\]
where $\delta_{1}=P_{n-1}(b;c_{2},c_{3},\ldots,c_{n})$, and
\[
\frac{1}{P_{m}(b;a,d_{2},d_{3},\ldots,d_{m})}-1=\frac{1}{\delta_{2}}
\left(\frac{P(a,b)}{P(b,a)}\right)\text{,}
\]
where $\delta_{2}=P_{m-1}(a;d_{2},d_{3},\ldots,d_{m})$.  Consequently
\begin{equation}\label{solvethis}
\left(\frac{1}{P_{n}(a;b,c_{2},c_{3},\ldots,c_{n})}-1\right)\left(\frac{1}{P_{m}(b;a,d_{2},d_{3},\ldots,d_{m})}-1\right)=\frac{1}{\delta_{1}\delta_{2}}\text{.}
\end{equation}
Solving (\ref{solvethis}) for $P_{m}(b;a,d_{2},d_{3},\ldots,d_{m})$, we obtain the following result.

\begin{prop}[\textbf{Distorted Difference Formula}]\label{ddf}
\[
P_{m}(b;a,d_{2},d_{3},\ldots,d_{m})=\frac{1-P_{n}(a;b,c_{2},c_{3},\ldots,c_{n})}{1+\bigl(\frac{1}{\delta_{1}\delta_{2}}-1\bigr)P_{n}(a;b,c_{2},c_{3},\ldots,c_{n})}\text{,}
\]
where $\delta_{1}=P_{n-1}(b;c_{2},c_{3},\ldots,c_{n})$ and $\delta_{2}=P_{m-1}(a;d_{2},d_{3},\ldots,d_{m})$.
\end{prop}
\noindent This formula may be viewed as an alternate generalization of James condition (c).  For example, taking $c_{2}=c_{3}=\ldots=c_{n}=b$ and $d_{2}=d_{3}=\ldots=d_{m}=a$,
we obtain
\begin{equation}\label{altc}
P_{m}(b;a,a,\ldots,a)=\frac{1-P_{n}(a;b,b,\ldots,b)}{1+(mn-1)P_{n}(a;b,b,\ldots,b)}\text{,}
\end{equation}
which similar in spirit to multi-James condition (D) for $k=1$.

In the single-opponent case, for $0<a<1$ and $0<b<1$, James condition (d) can be obtained by applying the involutive property three times and condition (c) twice (see \cite[Proposition 7]{hjm}).  The analogous argument, involving Proposition \ref{pseudo} and equation (\ref{altc}), gives rise to a different identity: the expression $P_{n}(a;b,b,\ldots,b)$ is equal to
\[
P_{n}\!\left(\frac{n^{2}(1-b)}{1+(n^{2}-1)(1-b)};\frac{n^{2}(1-a)}{1+(n^{2}-1)(1-a)},\ldots,\frac{n^{2}(1-a)}{1+(n^{2}-1)(1-a)}\right)\text{.}
\]
Combining this observation with condition (D), one obtains a weaker version of Proposition \ref{level}.

\section{Uniqueness of $P_{n}$}\label{S:odds}

Many of the formulae from Section \ref{S:formulae} allow us restate a multi-opponent James function $P_{n}$ in terms of the single-opponent James function $P=P_{1}$.  Much of the work in \cite{hjm} was devoted to finding new examples of Jamesian functions; that is, functions other than $P$ that satisfy the James conditions.  It might be reasonable to expect that substituting an arbitrary Jamesian function (or perhaps an arbitrary involutive Jamesian function) into one of the formulae from Section \ref{S:formulae} would lead to another class of functions that satisfy the multi-James conditions.  As it turns out, that hypothesis is entirely incorrect.\bigskip

Throughout this section, let $\{J_{n}\}$ denote a collection of functions that satisfy the multi-James conditions, as stated in Section \ref{S:formula}.
As part of this assumption, we stipulate that each $J_{n}$ has the same domain as the corresponding $P_{n}$ and has codomain $[0,1]$.
It suffices to restrict our attention to the situation where all the winning percentages belong to the interval $(0,1)$, as demonstrated by the following lemma.

\begin{lemma}\label{zerolem}
Let $\{J_{n}\}$ be a collection of functions that satisfy the multi-James conditions.
If
\[
 J_{n}(a;b_{1},b_{2},\ldots,b_{n})=P_{n}(a;b_{1},b_{2},\ldots,b_{n})
\]
whenever all the winning percentages belong to the interval $(0,1)$, then each
function $J_{n}$ is identical to the corresponding function $P_{n}$.
\end{lemma}

\begin{proof}
Conditions (B) and (F) allow us to disregard any of the $b_{i}$ that are $0$, so that
\[
 J_{n}(a;b_{1},b_{2},\ldots,b_{n})=P_{n}(a;b_{1},b_{2},\ldots,b_{n})
\]
whenever $0<a<1$ and $0\leq b_{i}<1$ for $1\leq i\leq n$, as long as at least one $b_{i}$ is nonzero.  
Hence, under these hypotheses,
\[
\lim_{a\rightarrow 0^{+}}J_{n}(a;b_{1},b_{2},\ldots,b_{n})=\lim_{a\rightarrow 0^{+}}P_{n}(a;b_{1},b_{2},\ldots,b_{n})=0
\]
and
\[
\lim_{a\rightarrow 1^{-}}J_{n}(a;b_{1},b_{2},\ldots,b_{n})=\lim_{a\rightarrow 1^{-}}P_{n}(a;b_{1},b_{2},\ldots,b_{n})=1\text{.}
\]
Thus condition (E) implies that $J_{n}(0;b_{1},b_{2},\ldots,b_{n})=0$ and $J_{n}(1;b_{1},b_{2},\ldots,b_{n})=1$ whenever
$0\leq b_{i}<1$, with at least one $b_{i}$ being nonzero.  Moreover, (E) implies that $J_{n}(1;0,0,\ldots,0)=1$.

Condition (C) implies that
\[
J_{n}(a;0,0,\ldots,0)=1-n\cdot J(0;a,0,0,\ldots,0)=1
\]
whenever $0<a<1$.  Likewise, if $0\leq a<1$ and exactly one term $b_{j}$ equals $1$, condition (C) implies that
\[
0\leq J_{n}(a;b_{1},b_{2},\ldots,b_{n})\leq 1-J_{n}(1;a,b_{1},b_{2},\ldots,b_{j-1},b_{j+1},\ldots,b_{n})=0\text{,}
\]
from which it follows that $J_{n}(a;b_{1},b_{2},\ldots,b_{n})=0$.  In other words, $J_{n}$ and $P_{n}$ agree for all values in their domain.
\end{proof}

For the remainder of this section, we will assume that all winning percentages belong to the interval $(0,1)$.
In this context, condition (E) implies that $0<J_{n}(a;b_{1},b_{2},\ldots,b_{n})<1$, so
we can define the sum formula and the substitution formula in the same manner as for the multi-opponent James functions:
\[
\frac{1}{J_{n}(a;b_{1},b_{2},\ldots,b_{n})}-1=\sum_{i=1}^{n}\frac{J(b_{i},a)}{J(a,b_{i})}
\]
and
\[
\frac{1}{J_{n}(a;b_{1},b_{2},\ldots,b_{n})}-1=\frac{J(c,a)}{J(a,c)}\left(\frac{1}{J_{n}(c;b_{1},b_{2},\ldots,b_{n})}-1\right)\text{,}
\]
where $J=J_{1}$ (see Propositions \ref{sumform} and \ref{subform}).  Likewise, the reduction formula may be defined
\[
\frac{1}{J_{n}(a;b_{1},b_{2},\ldots,b_{n})}-1=\frac{J(b_{1},a)}{J(a,b_{1})}\left(\frac{1}{J_{n-1}(b_{1};b_{2},b_{3},\ldots,b_{n})}\right)
\]
for $n\geq 2$ (see Proposition \ref{reductioprop}).  Our goal is to show that any one of these formulae, in the presence of the multi-James conditions, guarantees that $J_{n}=P_{n}$ for all $n$.\bigskip

Our next observation underlies all the subsequent results in this section.

\begin{lemma}\label{sublem}
Let $\{J_{n}\}$ be a collection of functions that satisfy the multi-James conditions.  If the substitution formula holds for $n=1$, then the function $J=J_{1}$
is identical to the James function $P=P_{1}$.
\end{lemma}

\begin{proof}
By hypothesis,
\[
\frac{1}{J(a,b)}-1=\frac{J(c,a)}{J(a,c)}\left(\frac{1}{J(c,b)}-1\right)=\frac{J(c,a)}{J(a,c)}\frac{J(b,c)}{J(c,b)}
\]
for all values of $a$, $b$, and $c$.  Condition (A) implies $J(a,\frac{1}{2})=a$ and $J(b,\frac{1}{2})=b$,
so taking $c=\frac{1}{2}$ yields the expression
\[
\frac{1}{J(a,b)}-1=\frac{b(1-a)}{a(1-b)}\text{.}
\]
Therefore
\[
J(a,b)=\frac{1}{1+\displaystyle\frac{b(1-a)}{a(1-b)}}=
\frac{a(1-b)}{a(1-b)+b(1-a)}=P(a,b)\text{,}
\]
as we had hoped to show.
\end{proof}

We now turn our attention to the sum formula.

\begin{prop}\label{jamessum}
Let $\{J_{n}\}$ be a collection of functions that satisfy the multi-James conditions.  If the sum formula holds for $\{J_{n}\}$, then $J_{n}=P_{n}$ for all $n$.
\end{prop}

\begin{proof}
Let $a$, $b$, and $c$ be arbitrary points in the interval $(0,1)$.  To simplify notation, define
\[
\alpha_{1}=\frac{J(b,a)}{J(a,b)}\text{,}\quad\alpha_{2}=\frac{J(a,c)}{J(c,a)}\text{,}
\quad\text{and}\quad\alpha_{3}=\frac{J(c,b)}{J(b,c)}\text{.}
\]
Condition (C), together with the sum formula, dictates that
\begin{align*}
1&=J_{2}(a;b,c)+J_{2}(b;a,c)+J_{2}(c;a,b)\\
&=\frac{1}{1+\alpha_{1}+\displaystyle\frac{1}{\alpha_{2}}}
+\frac{1}{1+\displaystyle\frac{1}{\alpha_{1}}+\alpha_{3}}
+\frac{1}{1+\alpha_{2}+\displaystyle\frac{1}{\alpha_{3}}}\\
&=1+\frac{-1-\alpha_{1}\alpha_{2}}{1+\alpha_{2}+\alpha_{1}\alpha_{2}}
+\frac{\alpha_{1}}{1+\alpha_{1}+\alpha_{1}\alpha_{3}}
+\frac{\alpha_{3}}{1+\alpha_{3}+\alpha_{2}\alpha_{3}}\\
&=1+\frac{-1+2\alpha_{1}\alpha_{2}\alpha_{3}-\alpha_{1}^{2}\alpha_{2}^{2}\alpha_{3}^{2}}
{(1+\alpha_{2}+\alpha_{1}\alpha_{2})(1+\alpha_{1}+\alpha_{1}\alpha_{3})(1+\alpha_{3}+\alpha_{2}\alpha_{3})}\\
&=1-\frac{(1-\alpha_{1}\alpha_{2}\alpha_{3})^{2}}{(1+\alpha_{2}+\alpha_{1}\alpha_{2})(1+\alpha_{1}+\alpha_{1}\alpha_{3})(1+\alpha_{3}+\alpha_{2}\alpha_{3})}
\text{.}
\end{align*}
Therefore $\alpha_{1}\alpha_{2}\alpha_{3}=1$, which implies that
\[
\frac{J(b,a)}{J(a,b)}=\frac{J(c,a)}{J(a,c)}\frac{J(b,c)}{J(c,b)}\text{.}
\]
Thus the substitution formula holds for $n=1$, so Lemma \ref{sublem} implies that $J=P$.  Consequently
\[
\frac{1}{J_{n}(a;b_{1},b_{2},\ldots,b_{n})}-1=\sum_{i=1}^{n}\frac{J(b_{i},a)}{J(a,b_{i})}
=\sum_{i=1}^{n}\frac{P(b_{i},a)}{P(a,b_{i})}=\frac{1}{P_{n}(a;b_{1},b_{2},\ldots,b_{n})}-1
\]
for all $n$, as we had hoped to show.
\end{proof}

Next we consider the substitution formula for arbitrary values of $n$.

\begin{prop}\label{jamesodd}
Let $\{J_{n}\}$ be a collection of functions that satisfy the multi-James conditions.  If the substitution formula holds for $\{J_{n}\}$,
then $J_{n}=P_{n}$ for all $n$.
\end{prop}

\begin{proof}

Since the substitution formula holds for all $n$, Lemma \ref{sublem} implies that $J=P$.
Our strategy for $n\geq 2$ will be to show that $J_{n}$ may be represented using the analog of the shifted sum formula (Proposition \ref{shift}):
\begin{equation}\label{representation}
\frac{1}{J_{n}(a;b_{1},b_{2},\ldots,b_{n})}-1=\frac{J(b_{1},a)}{J(a,b_{1})}\left(1+\sum_{i=2}^{n}\frac{J(b_{i},b_{1})}{J(b_{1},b_{i})}\right)\text{,}
\end{equation}
which will guarantee that $J_{n}=P_{n}$. We will employ backwards induction, not on $n$ but on the total number of
times the winning percentage $b_{1}$ appears among the opponents.

Consider, first of all, an expression of the form
\[
J_{n}(a;b_{1},b_{1},\ldots,b_{1})\text{,}
\]
where the same winning percentage $b_{1}$ appears $n$ times among the opponents.  Condition (C) dictates that
\[
J_{n}(b_{1};b_{1},b_{1},\ldots,b_{1})=\frac{1}{n+1}
\]
and hence
\[
\frac{1}{J_{n}(b_{1};b_{1},b_{1},\ldots,b_{1})}-1=n\text{.}
\]
Therefore the substitution formula shows that
\begin{align*}
\frac{1}{J_{n}(a;b_{1},b_{1},\ldots,b_{1})}-1&=\frac{J(b_{1},a)}{J(a,b_{1})}\left(\frac{1}{J_{n}(b_{1};b_{1},b_{1},\ldots,b_{1})}-1\right)\\
&=\frac{J(b_{1},a)}{J(a,b_{1})}\bigl(1+(n-1)\bigr)\\
&=\displaystyle\frac{J(b_{1},a)}{J(a,b_{1})}\left(1+\sum_{i=2}^{n}\frac{J(b_{1},b_{1})}{J(b_{1},b_{1})}\right)\text{.}
\end{align*}
Thus (\ref{representation}) represents $J_{n}$ whenever $b_{1}$ appears $n$ times among the opponents.

Now suppose that (\ref{representation}) represents $J_{n}$ whenever the winning percentage $b_{1}$ appears at least $k+1$ times among the opponents.
Consider the expression
\[
J_{n}(b_{1};\underbrace{b_{1},b_{1},\ldots,b_{1}}_{k\text{ terms}},b_{k+1},b_{k+2},\ldots,b_{n})\text{.}
\]
It follows from condition (C) that
\begin{equation}\label{sumhere}
(k+1)J_{n}(b_{1};b_{1},b_{1},\ldots,b_{1},b_{k+1},b_{k+2},\ldots,b_{n})=1-\sum_{j=k+1}^{n}R_{n}(j)\text{,}
\end{equation}
where
\[
R_{n}(j)=J_{n}(b_{j};\underbrace{b_{1},b_{1},\ldots,b_{1}}_{k+1\text{ terms}},b_{k+1},b_{k+2},\ldots,b_{j-1},b_{j+1},\ldots,b_{n})
\]
for $k+1\leq j\leq n$.  Since $R_{n}(j)$ represents a function for which $b_{1}$ appears at least $k+1$ times among the opponents,
the induction hypothesis applies to each $R_{n}(j)$.  In particular,
\begin{align*}
\frac{1}{R_{n}(j)}-1&=\frac{J(b_{1},b_{j})}{J(b_{j},b_{1})}\left(1+k\cdot\frac{J(b_{1},b_{1})}{J(b_{1},b_{1})}+\sum_{\substack{i=k+1,\\i\neq j}}^{n}\frac{J(b_{i},b_{1})}{J(b_{1},b_{i})}\right)\\
&=\frac{J(b_{1},b_{j})}{J(b_{j},b_{1})}\left(k+1-\frac{J(b_{j},b_{1})}{J(b_{1},b_{j})}+\sum_{i=k+1}^{n}\frac{J(b_{i},b_{1})}{J(b_{1},b_{i})}\right)\\
&=-1+\frac{J(b_{1},b_{j})}{J(b_{j},b_{1})}\left(k+1+\sum_{i=k+1}^{n}\frac{J(b_{i},b_{1})}{J(b_{1},b_{i})}\right)\text{.}
\end{align*}
Therefore
\[
R_{n}(j)=\frac{J(b_{j},b_{1})}{J(b_{1},b_{j})}\left(k+1+\sum_{i=k+1}^{n}\frac{J(b_{i},b_{1})}{J(b_{1},b_{i})}\right)^{-1}\text{,}
\]
so
\begin{align*}
1-\sum_{j=k+1}^{n}R_{n}(j)&=1-\frac{\displaystyle\sum_{j=k+1}^{n}\frac{J(b_{j},b_{1})}{J(b_{1},b_{j})}}{k+1+\displaystyle\sum_{i=k+1}^{n}\frac{J(b_{i},b_{1})}{J(b_{1},b_{i})}}\\
&=\frac{k+1}{k+1+\displaystyle\sum_{i=k+1}^{n}\frac{J(b_{i},b_{1})}{J(b_{1},b_{i})}}
\text{.}
\end{align*}
Thus it follows from (\ref{sumhere}) that
\begin{align*}
\frac{1}{J_{n}(b_{1};b_{1},b_{1},\ldots,b_{1},b_{k+1},b_{k+2},\ldots,b_{n})}-1
&=\frac{k+1}{1-\displaystyle\sum_{j=k+1}^{n}R_{n}(j)}-1\\
&=\left(k+1+\sum_{i=k+1}^{n}\frac{J(b_{i},b_{1})}{J(b_{1},b_{i})}\right)-1\\
&=k+\sum_{i=k+1}^{n}\frac{J(b_{i},b_{1})}{J(b_{1},b_{i})}\text{.}
\end{align*}
Consequently the substitution formula shows that
\begin{multline*}
\smash{\frac{1}{J_{n}(a;\underbrace{b_{1},b_{1},\ldots,b_{1}}_{k\text{ terms}},b_{k+1},b_{k+2},\ldots,b_{n})}-1}
=\frac{J(b_{1},a)}{J(a,b_{1})}\left(k+\displaystyle\sum_{i=k+1}^{n}\frac{J(b_{i},b_{1})}{J(b_{1},b_{i})}\right)\\
=\frac{J(b_{1},a)}{J(a,b_{1})}\left(1+\displaystyle\sum_{i=2}^{k}\frac{J(b_{1},b_{1})}{J(b_{1},b_{1})}+\sum_{i=k+1}^{n}\frac{J(b_{i},b_{1})}{J(b_{1},b_{i})}\right)
\text{.}
\end{multline*}
Hence condition (F) implies that (\ref{representation}) represents $J_{n}$ whenever $b_{1}$ appears $k$ times among the opponents.
In other words, our induction argument is complete.
\end{proof}

A similar result applies to the reduction formula.

\begin{prop}
Let $\{J_{n}\}$ be a collection of functions that satisfy the multi-James conditions.  If the reduction formula holds for $\{J_{n}\}$,
then $J_{n}=P_{n}$ for all $n$.
\end{prop}

\begin{proof}
Let $a$, $b$, and $c$ be arbitrary points in the interval $(0,1)$.  The reduction formula dictates that
\[
\frac{1}{J_{2}(a;b,c)}-1=\frac{J(b,a)}{J(a,b)}\left(\frac{1}{J(b,c)}\right)\text{.}
\]
Likewise,
\[
\frac{1}{J_{2}(a;c,b)}-1=\frac{J(c,a)}{J(a,c)}\left(\frac{1}{J(c,b)}\right)\text{.}
\]
Condition (F) dictates that $J_{2}(a;b,c)=J_{2}(a;c,b)$, from which it follows that
\[
\frac{J(b,a)}{J(a,b)}=\frac{J(c,a)}{J(a,c)}\frac{J(b,c)}{J(c,b)}\text{.}
\]
Therefore Lemma \ref{sublem} implies that $J=P$.   A straightforward induction argument, based on the reduction formula, shows that $J_{n}=P_{n}$ for all $n$.
\end{proof}

One can identify several other properties which, in conjunction with the multi-James conditions, imply that $J_{n}=P_{n}$.  The following
observation, which is based on the proof of Theorem 1 in \cite{luce}, pertains to a property discussed in Section \ref{S:formula}.

\begin{prop}
Let $\{J_{n}\}$ be a collection of functions that satisfy the multi-James conditions.  If $\{J_{n}\}$ possesses the property of independence from irrelevant alternatives, namely
\begin{equation}\label{jind}
\frac{J_{n}(b;a,c_{2},c_{3},\ldots,c_{n})}{J_{n}(a;b,c_{2},c_{3},\ldots,c_{n})}=
\frac{J(b,a)}{J(a,b)}
\end{equation}
whenever all the winning percentages  belong to the interval $(0,1)$, then $J_{n}=P_{n}$ for all $n$.
\end{prop}

\begin{proof}
Consider a set $\{a,b_{1},b_{2},\ldots,b_{n}\}$ of winning percentages in the interval $(0,1)$.  Applying (\ref{jind}), along with conditions (F) and (C), we see that
\begin{align*}
\sum_{i=1}^{n}\frac{J(b_{i},a)}{J(a,b_{i})}&=
\sum_{i=1}^{n}\frac{J_{n}(b_{i};a,b_{1},b_{2},\ldots,b_{i-1},b_{i+1},\ldots,b_{n})}{J_{n}(a;b_{i},b_{1},b_{2},\ldots,b_{i-1},b_{i+1},\ldots,b_{n})}\\
&=\frac{\displaystyle\sum_{i=1}^{n}J_{n}(b_{i};a,b_{1},b_{2},\ldots,b_{i-1},b_{i+1},\ldots,b_{n})}{J_{n}(a;b_{1},b_{2},\ldots,b_{n})}\\
&=\frac{1-J_{n}(a;b_{1},b_{2},\ldots,b_{n})}{J_{n}(a;b_{1},b_{2},\ldots,b_{n})}=\frac{1}{J_{n}(a;b_{1},b_{2},\ldots,b_{n})}-1\text{.}
\end{align*}
Thus the sum formula holds for $\{J_{n}\}$, so Proposition \ref{jamessum} guarantees that $J_{n}=P_{n}$.
\end{proof}

Another property, introduced in tandem with independence from irrelevant alternatives, gives
rise to a similar result.

\begin{prop}
Let $\{J_{n}\}$ be a collection of functions that satisfy the multi-James conditions.  If the odds ratio
\begin{equation}\label{jodds}
\frac{J_{m}(a;c_{1},c_{2},\ldots,c_{m})\bigl(1-J_{n}(a;b_{1},b_{2},\ldots,b_{n})\bigr)}{\bigl(1-J_{m}(a;c_{1},c_{2},\ldots,c_{m})\bigr)J_{n}(a;b_{1},b_{2},\ldots,b_{n})}
\end{equation}
is independent of $a$ whenever all the winning percentages belong to the interval $(0,1)$, then $J_{n}=P_{n}$ for all $n$.
\end{prop}

\begin{proof}
By assumption, the odds ratio
\[
\frac{J(a,c)\bigl(1-J_{n}(a;b_{1},b_{2},\ldots,b_{n})\bigr)}{\bigl(1-J(a,c)\bigr)J_{n}(a;b_{1},b_{2},\ldots,b_{n})}=\frac{J(a,c)}{J(c,a)}\left(\frac{1}{J_{n}(a;b_{1},b_{2},\ldots,b_{n})}-1\right)
\]
takes on the same value, no matter the value of $a$.  Taking $a=c$, we see that
\[
\frac{J(a,c)}{J(c,a)}\left(\frac{1}{J_{n}(a;b_{1},b_{2},\ldots,b_{n})}-1\right)=\frac{1}{J_{n}(c;b_{1},b_{2},\ldots,b_{n})}-1\text{,}
\]
which is equivalent to the substitution formula.  Hence Proposition \ref{jamesodd} implies that $J_{n}=P_{n}$.
\end{proof}

We may summarize the conclusions of this section in the following manner.

\begin{thm}
Let $\{J_{n}\}$ be a collection of functions that satisfy the multi-James conditions. Any one of the following conditions implies that $J_{n}=P_{n}$ for all $n$:
\begin{enumerate}
\item The sum formula holds for $\{J_{n}\}$.
\item The substitution formula holds for $\{J_{n}\}$.
\item The reduction formula holds for $\{J_{n}\}$.
\item The collection $\{J_{n}\}$ possesses the property of independence from irrelevant alternatives, as stated in (\ref{jind}).
\item The odds ratio (\ref{jodds}) is independent of $a$.
\end{enumerate}
\end{thm}

In closing, we remark that the proofs in this section did not require the full strength of all six multi-James conditions.  For example, condition (A) was only used to show that $J(a,\frac{1}{2})=a$ and condition (D) was not invoked at all.  We leave it to the reader to determine the minimal assumptions required for each result.

\section*{Acknowledgments}

We would like to thank Maximillian C.\ W.\ Bender for his many valuable suggestions, particularly relating to the structure of Section \ref{S:formulae}.

 \end{document}